\documentclass[11pt,a4paper,twoside]{article}
\usepackage{a4}
\usepackage{color}
\usepackage{bm, amsmath, amssymb, amsthm} 
\usepackage{graphicx}

\newcommand{\Sc}{{\mbox{\rm Sc}}}

\def\<{\langle}
\def\>{\rangle}
\def\eps{\varepsilon}
\def\NN{\mathbb{N}}
\def\ZZ{\mathbb{Z}}
\def\RR{\mathbb{R}}
\def\CC{\mathbb{C}}

\def\calf{\mathcal{F}}

\def\tr{\operatorname{Tr\,}}
\def\id{\operatorname{id\,}}
\def\Div{\operatorname{div}}

\def\vol{\operatorname{vol}}

\def\eq{\hspace*{-1.5mm}&=&\hspace*{-1.5mm}}

\def\dt{\partial_t}

\newtheorem{corollary}{Corollary}
\newtheorem{definition}{Definition}
\newtheorem{example}{Example}
\newtheorem{remark}{Remark}
\newtheorem{lemma}{Lemma}
\newtheorem{proposition}{Proposition}
\newtheorem{theorem}{Theorem}

\author{Vladimir Rovenski
\thanks{
        E-mail: rovenski@math.haifa.ac.il.
        Supported by the Marie-Curie actions grant EU-FP7-P-2010-RG, No.~276919.}
        \ and \
        Leonid Zelenko\thanks{
        E-mail: zelenko@math.haifa.ac.il}\\
        {\small Mathematical Department, University of Haifa}
}

\title{The mixed scalar curvature flow on a fiber bundle}
\begin{document}

\date{}

\maketitle


\begin{abstract}
We apply conformal flows of metrics restricted to the orthogonal distribution $D$ of a foliation
to study the question:
Which foliations admit a metric such that the leaves are totally geodesic and
the mixed scalar curvature is positive?
 Our evolution operator includes the integrability tensor of $D$,
and for the case of integrable orthogonal distribution the flow velocity is proportional to the mixed scalar curvature.
We~observe that the mean curvature vector $H$ of $D$ satisfies along the leaves
the forced Burgers equation,
this reduces to the linear Schr\"{o}dinger equation, whose potential
function is a certain ``non-umbilicity" measure of $D$.
 On order to show convergence of the solution metrics $g_t$ as $t\to\infty$,
we normalize the flow, and instead of a foliation consider a~fiber bundle $\pi: M\to B$ of a
Riemannian manifold $(M, g_0)$.
In this case, if the ``non-umbilicity" of $D$ is smaller in a sense then the ``non-integrability",
then the limit mixed scalar curvature function is positive.
For integrable $D$, we give examples with foliated surfaces and twisted products.


\end{abstract}

\section{Introduction}
\label{sec:intro}

\textbf{1.1. Totally geodesic foliations}.
Let $M^{n+p}$ be a connected manifold, endowed with a $p$-dimensional foliation $\calf$, i.e., a partition of $M$ into $p$-dimensional submanifolds.
A foliation $\calf$ on a Riemanian manifold $(M, g)$ is \textit{totally geodesic} if the leaves (of $\calf$) are totally geodesic submanifolds.
A~Riemannian metric $g$ on $(M,\calf)$ is called \textit{totally geodesic} if $\calf$ is totally geodesic w.\,r.\,t.~$g$.

The~simple examples of totally geodesic foliations are parallel circles or winding lines on a flat torus, and a Hopf field of great circles on the sphere $S^3$.
Totally geodesic foliations appear naturally as null-distributions (or kernels) in the study of manifolds with degenerate curvature-like tensors and  differential forms.
Totally geodesic foliations of codimension-one on closed non-negatively curved space forms are completely understood:
they are given by parallel hyperplanes in the case of a flat torus $T^n$
and they do not exist for spheres $S^n$. However, if the codimension is greater than
one, examples of geometrically distinct totally geodesic foliations are abundant, see survey in \cite{rov-m}.

One of the principal problems of geometry of foliations reads as follows:
 \textit{Given a foliation $\calf$ on a manifold $M$ and a geometric property (P), does there exist a Riemannian metric $g$ on $M$ such that $\calf$ enjoys (P) with respect to $g$?}
Such problems (first posed by H.\,Gluck
for geodesic foliations) were studied already in the 1970's when D.\,Sullivan
provided a topological condition (called {\it topological tautness}) for a foliation, equivalent to
the existence of a Riemannian metric making all the leaves minimal.
 In~recent decades, several tools providing results of this sort have been developed. Among them, one may find Sullivan's
{\it foliated cycles} and new \textit{integral formulae}, \cite[Part 1]{rw-m},
\cite{wa1} etc.

\vskip1mm
\textbf{1.2. The mixed scalar curvature}.
There are three kinds of curvature for a foliation: 
tangential, transversal, and mixed (a plane that contains a tangent vector to the foliation and a vector orthogonal to it is said to be mixed).
 The geometrical sense of the mixed curvature follows from the fact that for a totally geodesic
foliation, certain components of the curvature tensor
(i.e., contained in the Riccati equation for the conullity tensor, see Section~\ref{subsec:prel}),
regulate the deviation of the leaf geodesics.
In general relativity, the \textit{geodesic deviation equation} is an equation involving the Riemann curvature tensor, which measures the change in separation of neighboring geodesics or, equivalently, the tidal force experienced by a rigid body moving along a geodesic.
In the language of mechanics it measures the rate of relative acceleration of two particles moving forward on neighboring geodesics.
Let $\{e_i,\,\eps_\alpha\}_{i\le n,\,\alpha\le p}$ be a local orthonormal frame on $TM$ adapted to $D$ and $D_\calf$.
 Tracing the Riccati equation yields the equation with
the \textit{mixed scalar curvature} that is the following function on $M$, see \cite{rov-m, rw-m, wa1}:
\[
 \Sc_{\,\rm mix} =\sum\nolimits_{i=1}^n\sum\nolimits_{\alpha=1}^p K(e_i, \eps_\alpha),
\]
where
$K(e_i, \eps_\alpha)$ is the sectional curvature of the mixed plane spanned by the vectors $e_i$ and $\eps_\alpha$.
For~example, $\Sc_{\,\rm mix}$ of a foliated surface $(M^2,g)$ is the gaussian curvature $K$.

Denote $(\,\cdot\,)^\calf$ and $(\,\cdot\,)^\bot$ projections onto $D_\calf$ and $D$, respectively.
The second fundamental tensor $b$ and the integrability tensor $T$ of $D$~are given by
\begin{eqnarray*}
 2\,b(X,Y) = (\nabla_X Y+\nabla_Y X)^\calf,\qquad
 2\,T(X,Y)=[X,\,Y]^\calf\qquad (X,Y\in D),
\end{eqnarray*}
where $\nabla$ is the Levi-Civita connection of~$g$.
For general (i.e., non-integrable) distribution $D$, define the domain $U_T=\{x\in M: T(x)\ne0\}$.
The~mean curvature vector of $D$ is given by $H=\tr_g b$.

For a totally geodesic foliation $\calf$ we have, see \cite{ran} and \cite{wa1}:
\begin{equation}\label{eq-ran}
 {\rm Sc}_{\rm mix} =\Div H +|H|^2+\|T\|^2 -\|b\,\|^2.
\end{equation}
By the Divergence Theorem, (\ref{eq-ran}) yields the integral formula with ${\rm Sc}_{\rm mix}$ on a closed manifold $M$.

A~foliation $\calf$ is \textit{conformal}, \textit{transversely harmonic}, or \textit{Riemannian},
if $b=\frac1nH\hat g$, $H=0$ or $b=0$, respectively.
In these cases, the distribution $D$ is called \textit{totally umbilical}, \textit{harmonic} or \textit{totally geodesic}, respectively.
Conformal foliations were introduced by Vaisman \cite{v79} as foliations
admitting a transversal conformal structure.
Molino developed a theory of Riemannian foliations on compact manifolds,
such foliations form a subclass of conformal foliations.

\begin{remark}\rm
Formula (\ref{eq-ran}) gives us decomposition criteria for foliated manifolds
(with an integrable orthogonal distribution) under the constraints on the sign of $\Sc_{\,\rm mix}$,
see \cite{wa1} and a survey in \cite{rov-m}:

(1) If $\calf$ and $\calf^\perp$ are complementary orthogonal totally umbilical and totally geodesic foliations on a closed
oriented Riemannian manifold $M$ with ${\rm Sc}_{\rm mix}\ge 0$, then $M$ splits along the foliations.

(2) A compact minimal foliation $\calf$ on a Riemannian manifold $M$ with an integrable orthogonal
distribution and $\Sc_{\,\rm mix} \ge 0$ splits  along the foliations.

(3) A minimal foliation $\calf$ on a Riemannian manifold $M$ with the integrable
orthogonal distribution and $\Sc_{\,\rm mix} > 0$ has no compact leaves.
\end{remark}

The basic question that we want to address in the paper is the following.

\noindent
\textbf{Question 1}: {Which foliations admit a totally geodesic metric of positive mixed scalar curvature}?

\begin{example}\label{R-03}\rm
(a) A change of initial metric along orthogonal distribution $D$ preserves the property
``$\calf$ is totally geodesic", see Lemma~\ref{L-nablaNN} in Section~\ref{subsec:tvarb}.
Let $\pi: M\to B$ be a fiber bundle with compact fibers.
One may deform the metric $g$ along $D$ to obtain a bundle-like totally geodesic metric~$\tilde g$
(which in general is not $D$-conformal to $g$) on a fiber or its neighborhood.
If there is a section $\xi:B\to M$ transversal to fibers,
then the deformation can be done globally,
and $\pi$ becomes a Riemannian submersion with totally geodesic fibers.
In this case, the mixed sectional curvature is non-negative (due to the formula $K(X,V)|X|^2|V|^2=|A_X V|^2$ for the mixed sectional curvature by O'Neill),
moreover, $\Sc_{\,\rm mix}$ with respect to $\tilde g$ is positive on
$U_T$ ($\Sc_{\,\rm mix}\equiv0$ when $D$ is integrable).

(b) For any $n\ge2$ and $p\ge1$ there exists a fiber bundle with
a closed $(n+p)$-dimensional total space and a compact $p$-dimensional fiber,
having a totally geodesic metric of positive mixed scalar curvature.
 To show this, consider the Hopf fibration $\tilde\pi:S^{3}\to S^2$ of a unit sphere $S^{3}$ by great circles (closed geodesics). Let $(\tilde F, g_1)$ and $(\tilde B, g_2)$ be closed Riemannian manifolds
with dimensions, respectively, $(p-1)$ and $(n-2)$.
 Let $M=\tilde F\times S^{3}\times\tilde B$ be the metric product, and $B=S^{2}\times\tilde B$.
Then $\pi:M\to B$ is a fibration with a totally geodesic fiber $F=\tilde F\times S^1$.
Certainly, $\Sc_{\,\rm mix}=2>0$.
\end{example}

Motivating by Remark~\ref{R-03}, we ask the following (more particular than \textbf{Question 1}).

\noindent
\textbf{Question 2}:\textit{
Given a Riemannian manifold $(M,g)$ with a totally geodesic foliation $\calf$,
does there exist a $D$-conformal to $g$ metric $\bar g$ on $M$ such that $\overline{\rm Sc}_{\,\rm mix}$ is positive?}

In the paper (at least in main results) we impose the additional restrictions:

 -- instead of a foliation, $M$ is a total space of a smooth fiber bundle $\pi: M\to B$,

 -- the fibers (leaves) are compact.

\noindent
Although a fiber bundle is locally a product (of the base and the fiber), this is not true globally.
Meanwhile, in Section~\ref{sec:main-res} we have example of solutions in the class of twisted products.

\textbf{1.3. $D$-conformal flows of metrics}.
We attack the \textbf{Question 2} using evolution PDEs.
Evolution equations are important tool to study physical and natural phenomena.
The~prototype for non-linear advection-diffusion processes is the \textit{Burgers equation}
 $v_{,t}+(v^2)_{,x} =\nu\,v_{,xx}$ for~a scalar function $v$
(a constant $\nu>0$ is the kinematic viscosity).
It serves as the simplest model equation for solitary waves, and is used for describing wave processes
in gas and fluid dynamics~\cite{s2008}.

A~\textit{geometric flow} (GF) of metrics on a manifold $M$ is a solution $g_t$ of an evolution equation  \begin{equation*}
 \dt g = S(g)\,,
\end{equation*}
where a geometric functional $S$ (a symmetric $(0,2)$-tensor) is usually related to some kind of curvature.
The theory of GFs is a new subject, of common interest in mathematics and physics.
GFs (e.g., the \textit{Ricci flow} and the \textit{Yamabe flow}),
correspond to dynamical systems in the infinite-dimensional space of all possible Riemannian metrics on a manifold.

The~notion of the $D$-\emph{truncated} $(r,k)$-\emph{tensor} $\hat S$ ($r=0,1$) will be helpful:
\[
 \hat S(X_1,\dots,X_k) = S(X_1^\bot,\dots,X_k^\bot)\qquad (X_i\in TM).
\]

Rovenski and Walczak \cite{rw-m} introduced $D$-truncated flows of metrics on codimensi\-on-one foliations, 
depending on the
extrinsic geometry of the leaves and posed the question:

 \textit{Given a foliation $(M,\calf)$ and a geometric property (P),
does there exist a $D$-truncated tensor $S$ such that solution metrics $g_t\ (t\ge0)$ to
$\dt g = \hat S(g)$
converge to a metric $g_\infty$ for which $\mathcal{F}$ enjoys~(P)}?

Some of results in \cite{rw-m}
were extended by the first author \cite{rw4} for GFs related to parabolic PDEs,
applications to the problem of prescribing the mean curvature function of a codimension-one foliation,
and examples with harmonic and totally umbilical foliations are given.

Rovenski and Wolak \cite{rovwol} studied the $D$-conformal flow of metrics
on a foliation of any codimension with the speed
proportional to the $\calf$-divergence of $H$.
Based on known long-time existence results for the heat flow they showed convergence of a solution to a metric for which
$H=0$;
actually under some topological assumptions they prescribe the mean curvature $H$.
The~conditions in their results are rather different and seem to be stronger of known hypotheses to guarantee tautness
in non-constructive results. This makes sense because it is much harder to provide the good GF.

For $D$-conformal flows of metrics on foliations, we have $\hat S(g)=s(g)\,\hat g$,
where $s(g)$ is a smooth function on the space of metrics on $M$, and the $D$-truncated metric tensor $\hat g$ is given~by
 $\hat g(N,\cdot)=0$ and $\hat g(X_1,X_2)=g(X_1,X_2)$ for all
 $X_i\in D,\ N\in D_\calf$.
By Lemma~\ref{L-nablaNN} in Section~\ref{subsec:tvarb}, $D$-conformal variations of a metric $g$ preserve the property ``$\calf$ is totally geodesic".

In~the paper we continue studying $D$-conformal GFs on a foliated manifold $(M,\calf)$ and,
in order to prescribe the positive $\Sc_{\,\rm mix}$, introduce the following flow of totally geodesic metrics:
\begin{equation}\label{E-GF-Ricmix-mu}
 \dt g = -2\,(\Sc_{\,\rm mix} -\|T\|^2 -\Phi)\,\hat g.
\end{equation}
Here $\Phi:M\to\RR$ is an arbitrary function constant along the leaves, it is used 
to normalize the flow equation
in order to obtain the convergence of solution metrics $g_t$ as $t\to\infty$.
 For integrable distribution $D$, we have the PDE (\ref{E-GF-Ricmix-mu}) with $T=0$.

In~order to prescribe the positive mixed scalar curvature, we examine the following.

\noindent
\textbf{Question 3}:
\textit{Given a Riemannian manifold $(M,g)$ with a totally geodesic foliation $\calf$,
when do the solution metrics $g_t$ of (\ref{E-GF-Ricmix-mu}) converge as $t\to\infty$
to a metric $\bar g$ for which $\overline{\rm Sc}_{\,\rm mix}$ is positive}?

\begin{example}\label{Ex-surfrev}\rm
(a) For a Hopf fibration $\pi:S^{2n+1}\to\CC P^n$ with fiber $S^1$, the orthogonal
distribution $D$ is non-integrable while it is totally geodesic ($T\ne0,\,b=0$).
The standard metric on $S^{2n+1}$ is a fixed point of (\ref{E-GF-Ricmix-mu}) with $\Phi=0$,
and we have
$\Sc_{\,\rm mix}=\|T\|^2>0$ (see Lemma~\ref{L-CC-riccati} in Section~\ref{subsec:prel}).

(b) Let $(M^2, g_0)$ be a surface (i.e., $\dim M=2$)
foliated by geodesics, and $K$ its Gaussian curvature.
Since one-dimensional distribution $D$ is integrable, we have $T=0$.
Then one may take $\Phi=0$, and reduce (\ref{E-GF-Ricmix-mu}) to the following~form:
\begin{equation}\label{E-GF-KM2}
 \dt g = -2\,K\,\hat g.
\end{equation}
The known Cole-Hopf transformation reduces the non-linear Burgers equation to the linear heat equation.
We observe that (\ref{E-GF-KM2}) yields both above mentioned PDEs for foliated surfaces.
Let $M_t^2\subset\RR^3$ be a smooth family of surfaces of revolution about the $Z$-axis,
\begin{equation*}
 r(x,\theta, t)=[\rho(x,t)\cos\theta,\ \rho(x,t)\sin\theta,\ h(x,t)]\qquad (0\le x\le l,\ -\pi\le \theta\le \pi).
\end{equation*}
Let the profile curves be the leaves of $\calf$ (geodesics).
Denote by $k$ the geodesic curvature of parallels (circles orthogonal to the leaves).
 The following properties are equivalent (see also Theorem~\ref{T-twisted1}):

 (i) The induced metrics $g_t$ (on $M_t^2$) are the solution of (\ref{E-GF-KM2}).

 (ii) The distance $\rho>0$ from the profile curve to the axis satisfies the heat equation
 $\rho_{,t}=\rho_{,xx}$.

 (iii) The geodesic curvature $k$ of parallels satisfies the Burgers equation
  $k_{,t}+(k^2)_{,x} =k_{,xx}$.
\end{example}

\noindent
\textbf{1.4. The structure of the paper}.
Section~\ref{sec:intro} introduces the $D$-conformal flow of metrics
and Section~\ref{sec:main-res} collects main results concerning Questions 2--3.
Section~\ref{sec:tvar} represents $D$-conformal variations of geometric quantities
and
Section~\ref{sec:egf} proves main results and applications to foliated surfaces.
Section~\ref{R-burgers-heat}
(Appendix) deals with linear parabolic PDEs on a closed Riemannian manifold
(its results seem to be known, but for the convenience of a reader we prove them).

\section{Main results}
\label{sec:main-res}

An important foundational step in the study of any system of evolutionary
PDEs is to show short-time existence and uniqueness.

\begin{proposition}\label{T-main-loc}
Let $\calf$ be a totally geodesic foliation of a closed Riemannian manifold $(M, g_0)$.
Then (\ref{E-GF-Ricmix-mu}) has a unique (smooth along the leaves) solution $g_t$ defined on a positive time interval~$[0,\eps)$.
\end{proposition}

We denote $\nabla^\calf f:=(\nabla f)^\calf$. Given a vector field $X$ and a function $F$ on $M$,
define the functions using the derivatives along the leaves: the divergence $\Div^\calf X=\sum_{\alpha=1}^{p} g(\nabla_\alpha X, \eps_\alpha)$ and the Laplacian $\Delta_\calf F=\Div^\calf(\nabla^\calf F)$. Remark that
$\nabla^\calf, \Div^\calf$ and $\Delta_\calf$ (i.e., along the leaves) are $t$-independent.

Notice the ``linear algebra" inequality $n\,\|b\,\|^2\ge |H|^2$
with the equality when the distribution $D$ is totally umbilical.
 Consider the following non-negative measure of ``non-umbilicity" of $D$:
\[
 \beta_D:=n^{-2}\big(n\,\|b\,\|^2 -|H|^2\big).
\]
For $p=1$, we have $\beta_D=n^{-2}\sum_{i<j}(k_i-k_j)^2$, where $k_i$ are principal curvatures of~$D$, see \cite{wa2}.

 By Proposition~\ref{R-consumb} in Section~\ref{subsec:tvarb}, (\ref{E-GF-Ricmix-mu}) preserves $\beta_D$
and the Schr\"{o}dinger operator
 $\mathcal{H}=-\Delta_\calf -\beta_D\id$
on the leaves. Recall that $-\Delta_\calf\ge0$.
Let $\lambda_0$ be the smallest real eigenvalue of $\mathcal{H}$
with a positive eigenfunction $e_0$ of unit $L^2$-norm.
Notice that $\lambda_0$ is constant on the leaves, and
if the leaves are compact then $\lambda_0(x)\ge-\max\limits_{F(x)}\beta_D$,
where $F(x)$ is the leaf through $x$.
To~show this, define on $F(x)$ the operator $\tilde{\mathcal{H}}=-\Delta_\calf -(\max\limits_{F(x)}\beta_D)\id$.
From $\mathcal{H}\ge\tilde{\mathcal{H}}\ge-(\max\limits_{F(x)}\beta_D)\id$
the estimate of $\lambda_0$ follows.

Denote by $[g_0]$ the $D$-conformal class of the metric $g_0$ on a foliated manifold $(M,\calf)$.
Certainly, $\beta_D$ and the operator $\mathcal{H}$ depend on $[g_0]$ only.

In Proposition~\ref{T-mainA} we observe that (\ref{E-GF-Ricmix-mu}) yields the multidimensional Burgers equation (for the mean curvature vector $H$)
 on any leaf,
and has a single-point global attractor when the leaf is compact.

\begin{proposition}\label{T-mainA}
Let $\calf$ be a totally geodesic foliation of a Riemannian manifold $(M, g_0)$,
and $H_0=-n\nabla^\calf(\log u_0)$ for some positive function $u_0$ on $M$.
If the  metrics $g_t\ (t\ge0)$ satisfy (\ref{E-GF-Ricmix-mu}) then

 (i) The mean curvature vector of $D$ w.r.t. $g_t$ is a unique solution of the forced Burgers PDE
\begin{equation}\label{Ec-ANtau1}
 \dt H +\nabla^\calf g(H,H)= n\,\nabla^\calf(\Div^\calf H) -n^2\,\nabla^\calf\beta_D,
\end{equation}

 (ii) $H$ converges exponentially as $t\to\infty$ on compact leaves to the vector field
  $\bar H{=}-n\nabla^\calf(\log e_0)$.
\end{proposition}

 The central result of the work is Theorem~\ref{T-main0} (and its applications), where
 for convergence of solution $g_t$ as $t\to\infty$ on a fiber bundle,
 the flow of metrics (\ref{E-GF-Ricmix-mu}) is normalized (taking $\Phi=n\,\lambda_0$).

\begin{theorem}\label{T-main0}
Let $\pi: (M,g_0)\to B$ be a fibration with compact totally geodesic fibers,
and $H_0=-n\nabla^\calf(\log u_0)$ for a positive function $u_0$ on $M$ (the potential).
Then

(i) The PDE (\ref{E-GF-Ricmix-mu}) admits a unique global smooth solution $g_t\ (t\ge0)$ on $M$.

(ii) Let $\Phi=n\,\lambda_0$. Then the metrics $g_t$ converge as $t\to\infty$ to the limit metric $\bar g$.
There exists real $\tilde C>0$ depending on $[g_0]$ such that
$\overline{\Sc}_{\,\rm mix}>0$ when
$\|T(x)\|^2_{g_0}>\tilde C\max\nolimits_{F(x)}\beta_D$ for all $x\in M$.
\end{theorem}

\begin{corollary}\label{C-01}
In conditions of Theorem~\ref{T-main0} we have the following:

(a) If $\nabla^\calf\beta_D=0$ then $\calf$ is $\bar g$-transversely harmonic
and $\overline{\Sc}_{\,\rm mix}\ge0$, while $\overline{\Sc}_{\,\rm mix}>0$ on~$U_T$.

(b) If $\beta_D=0$ then $\calf$ is $\bar g$-Riemannian and $\overline{\Sc}_{\,\rm mix}$ is constant on the fibers;

\quad moreover, if $D$ is integrable, then $\overline{\Sc}_{\,\rm mix}=0$ and $M$ is locally
the product with respect to $\bar g$

\quad
(the product globally for simply connected~$M$).
\end{corollary}

\begin{remark}\rm
 The condition $H=\nabla^\calf(\log u_0)$ of Theorem~\ref{T-main0} and Proposition~\ref{T-mainA}
 is satisfied for twisted products (see Theorem~\ref{T-twisted1} in what follows).
\end{remark}

We consider applications of Theorem~\ref{T-main0} to the twisted products.
The particular case of surfaces of revolution is studied in Section~\ref{sec:surfrev}
(see also Example~\ref{Ex-dt_rN} in Section~\ref{subsec:evolv}).

\begin{definition}[see \cite{pr}]\label{D-twisted}\rm
Let $(M_1,g_1)$ and $(M_2,g_2)$ be Riemannian manifolds,
and $f\in C^\infty(M_1\times M_2)$ a positive function.
The \textit{twisted product} $M_1\times_{f} M_2$ is
the manifold $M=M_1\times M_2$ with the metric $g=(f^2 g_1)\oplus g_2$.
If the warping function $f$ depends on $M_2$ only, then we have a \textit{warped product}.
\end{definition}

\begin{remark}\rm
The fibers $M_1\times\{y\}$ of a twisted product are umbilical with the \textit{mean curvature vector} $H=-\nabla^\calf(\log f)$, while $\{x\}\times M_2$ are totally geodesic.
The fibers of $M_1\times\{y\}$ have

(a) constant mean curvature if and only if $\|\nabla^\calf(\log f)\|$ is a function of $M_1$, and

(b) parallel mean curvature vector if and only if $f=f_1 f_2$ for some $f_i:M_i\to\RR_+\ (i=1,2)$.

\noindent
If on a simply connected complete Riemannian manifold $(M,g)$ two orthogonal foliations
with the properties (a)--(b) are given, then $M$ is a twisted product, see~\cite{pr}.
\end{remark}

 If $D$ is integrable and totally umbilical (i.e., $\beta_D=T=0$),
 then $\Phi=0$, hence (\ref{E-GF-Ricmix-mu}) is reduced to the $D$-conformal flow of metrics
 with the speed proportional to $\Sc_{\,\rm mix}$:
\ $\dt g = -2\,\Sc_{\,\rm mix}\,\hat g$.

\begin{theorem}\label{T-twisted1}
 Let $(M,g_t)=M_1\times_{f_t} M_2$ be a family of twisted products of Riemannian manifolds
 $(M_1, g_1)$ and $(M_2, g_2)$. Then the following properties are equivalent:

 (i) The metrics $g_t$ satisfy the evolution equation (\ref{E-GF-Ricmix-mu}).

 (ii) The mean curvature vector of fibers $M_1\times\{y\}$ satisfies the Burgers type PDE
\begin{equation}\label{E-multi-B}
 \dt H +\nabla^\calf g(H,H)= n\,\nabla^\calf(\Div^\calf H).
\end{equation}

 (iii) The warping function satisfies the heat equation $\dt f=n\,\Delta_\calf f$.
\end{theorem}

\begin{corollary}\label{T-main2}
Let $M_1\times_{f} M_2$ be a twisted product of closed Riemannian manifolds $(M_1, g_1)$ and $(M_2,g_2)$
for a positive warping function $f\in C^\infty(M_1\times M_2)$.
Then (\ref{E-GF-Ricmix-mu}) admits a unique smooth solution $g_t\ (t\ge0)$,
consisting of twisted product metrics on $M_1\times_{f_t} M_2$.
As $t\to\infty$, the metrics $g_t$ converge to the metric $\bar g$ of the product
$(M_1,\bar f^{\,2} g_1)\times(M_2,g_2)$, where $\bar f(x)=\int_{M_2} f(0, x, y)\,d y_g $.
\end{corollary}

\section{Auxiliary results}
\label{sec:tvar}

\subsection{Preliminaries}
\label{subsec:prel}

For the convenience of a reader, we recall some definitions.

\begin{definition}[see \cite{rov-m}]\rm
A family ${\mathcal F}=\{F_\alpha\}_{\alpha\in B}$ of connected subsets of a manifold $M^{n+p}$
is said to be a $p$-\textit{dimensional foliation}, if
\,1)~$\bigcup_{\alpha\in B} F_\alpha=M^{n+p}$,
\,2)~$\alpha\ne\beta\Rightarrow F_\alpha\bigcap F_\beta=\emptyset$,
\,3)~for any point $x\in M$ there exists a $C^r$-chart
$\varphi_x: U_x\to\RR^p$ such that $y\in U_x$, $\varphi_x(y)=0$, and if
$\ U_x\bigcap F_\alpha\ne\emptyset$ the connected components of the sets
$\varphi_x(U_x\bigcap F_\alpha)$ are given by equations
$x_{p+1} = c_{p+1},\ \ldots,\ x_{n+p}=c_{n+p},$ where $c_j$'s are constants.
The sets $F_\alpha$ are immersed submanifolds of $M$ called \textit{leaves} of $\calf$.
\end{definition}

\begin{definition}\rm
Let $F$ and $B$ be smooth manifolds.
A \textit{fiber bundle} over $B$ with fiber $F$ is a smooth manifold $M$,
together with a surjective submersion $\pi: M\to B$ satisfying a local triviality condition:
For any $x\in B$ there exists an open set $U$ in $B$ containing $x$,
and a diffeomorphism $\phi:\pi^{-1}(U)\to U\times F$ (called a local trivialization)
such that $\pi=\pi_1\circ\phi$ on $\pi^{-1}(U)$, where $\pi_1(x,y)=x$ is the projection on the first factor. The fiber at $x$, denoted by $F_x$, is the set $\pi^{-1}(x)$, which is diffeomorphic to $F$ for each $x$.
We call $M$ the total space, $B$ the base space and $\pi$ the projection.
\end{definition}

The Weingarten operator $A_N$ of $D$ w.\,r.\,t. to $N\in D_\calf$ and the
operator $T^\sharp_N$ are given~by
\begin{equation*}
 g(A_N(X),\,Y) = g(b(X,Y),\ N),\qquad
 g(T^\sharp_N(X),\,Y) = g(T(X,Y),\ N),\qquad (X,Y\in D).
\end{equation*}
The~\textit{co-nullity operator} $C: D_\calf\times D\to D$ (see for example \cite{rov-m}) is defined~by
\begin{equation*}
  C_N(X)=-(\nabla_{\!X} N)^\bot\quad (X\in D,\ N\in D_\calf).
\end{equation*}
Hence $C_N=A_N+T^\sharp_N$ -- the linear operator on orthogonal distribution $D$.
The equality $C=0$ means that $D$ is integrable and the integral manifolds are totally geodesic in $M$.

Define the self-adjoint $(1,1)$-tensor $R_N=R(N,\cdot)N\ (N\in D_\calf)$ on $D$, called Jacobi operator.
Next lemma represents (\ref{eq-ran}) in equivalent form as (\ref{E-RicNs}).

\begin{lemma}\label{L-CC-riccati}
For a totally geodesic foliation,
we have
\begin{eqnarray}\label{E-riccati}
 \nabla_N\,C_N=C_N^2+R_N\qquad(N\in D_\calf),\\
\label{E-RicNs}
 \Sc_{\,\rm mix}  -\|T\|^2 = \Div^\calf H -g(H, H)/n -n\,\beta_D.
\end{eqnarray}
\end{lemma}

\begin{proof} For Riccati equation (\ref{E-riccati}) see \cite{rov-m}.
 Substituting $C_N=A_N+T_N^\sharp$ into (\ref{E-riccati}) and taking the symmetric and skew-symmetric parts, yield
 a pair of equations
\begin{equation}\label{E-nablaT-1}
 \nabla_N A_N = A_N^2 +(T_N^\sharp)^2 +R_N,\qquad
 \nabla_N T_N^\sharp = A_N\,T_N^\sharp +T_N^\sharp A_N.
\end{equation}
By (\ref{E-nablaT-1})$_2$, we also have
\begin{equation}\label{E-nablaT-1b}
 -N(\|T\|^2) =2\tr(T_N^\sharp(\nabla_N\,T_N^\sharp)) =4\tr(A_N(T_N^\sharp)^2).
\end{equation}
Let $e_i$ be a local orthonormal frame of $D$. Since the leaves are totally geodesic, we have
$\nabla_N e_i\in D$ for any vector $N\in D_\calf$.
Then, using $g(A_N(e_i),\nabla_N e_i)=0$, we find
\begin{equation*}
 N(\tr A_N)
 =\sum\nolimits_i\big[\,g((\nabla_N A_N)(e_i),e_i) +g(A_N(\nabla_N e_i),e_i) +g(A_N(e_i),\nabla_N e_i)\big] = \tr(\nabla_N A_N).
\end{equation*}
Hence, the contraction of (\ref{E-nablaT-1})$_1$ over $D$ yields the formula
\begin{equation}\label{E-nablaT-2}
 N(\tr A_N) = \tr(A_N^2)+\tr((T^\sharp_N)^2)+\sum\nolimits_{j=1}^n K(e_j, N)
\end{equation}
for any unit vector $N\in D_\calf$. Note that $\tr A_N=g(H,N)$. We have
\begin{equation*}
 \sum\limits_{\alpha=1}^p \eps_\alpha (\tr A_{\eps_\alpha})=\Div^\calf H,\quad
 \sum\limits_{\alpha=1}^p\tr((T^\sharp_{\eps_\alpha})^2)=-\|T\|^2,\quad
 \sum\limits_{\alpha=1}^p\tr(A_{\eps_\alpha}^2)=\|b\,\|^2 = n\,\beta_D +\frac1n\,g(H,H).
\end{equation*}
Hence, the contraction of (\ref{E-nablaT-2}) over $D_\calf$ yields (\ref{E-RicNs}).
\end{proof}

\begin{remark}\label{R-intbeta}\rm
By the Divergence Theorem, from (\ref{E-RicNs}) and $\Div H = \Div^\calf H-g(H, H)$ we obtain
\begin{equation*}
 n\int_M\beta_D\,d\vol =\big(1-\frac1n\,\big)\int g(H, H)\,d\vol-\int(\Sc_{\,\rm mix}  -\|T\|^2)\,d\vol
 \ge -\int \Sc_{\,\rm mix} \,d\vol
\end{equation*}
on a closed $M$. Hence the inequality $\Sc_{\,\rm mix} <0$ yields that $\beta_D$ is somewhere positive.
\end{remark}

\subsection{$D$-truncated families of metrics}
\label{subsec:tvarb}

Since the difference of two connections is a tensor, $\dt\nabla^t$ is a $(1,2)$-tensor on $(M,g_t)$.
Differentiating the known formula for the {Levi-Civita connection}
with respect to $t$ yields
\begin{equation}\label{eq2}
 2\,g_t((\dt\nabla^t)(X, Y), Z)=(\nabla^t_X S)(Y,Z)+(\nabla^t_Y S)(X,Z)-(\nabla^t_Z S)(X,Y)
\end{equation}
for all $t$-independent vector fields $X,Y,Z\in\Gamma(TM)$, see~\cite{rw-m}.

\begin{lemma}
\label{L-nablaNN}
$D$-truncated variations of metrics preserve the property ``$\calf$ is totally geodesic".
\end{lemma}

\begin{proof}
Let $g_t\ (t\ge0)$ be a family of metrics on a foliation $(M,\calf)$ such that the tensor $S_t=\dt g_t$
is $D$-truncated. We claim that the second fundamental tensor of $\calf$ is evolved as
 $\dt {b^\bot} = -S^\sharp\circ\,{b^\bot}$.

Using (\ref{eq2}) and that the tensor $S$ is $D$-truncated, we find for $X\in D$ and $\xi,\eta\in D_\calf$,
\begin{eqnarray*}
  2\,g_t(\dt{b^\bot}(\xi,\eta),X)\eq g_t(\dt(\nabla^t_{\xi}\,\eta)+\dt(\nabla^t_{\eta}\,\xi),\ X)\\
 \eq(\nabla^t_{\xi} S)(X,\eta)+(\nabla^t_{\eta} S)(X,\xi)-(\nabla^t_X S)(\xi,\eta)
  =-2\,S({b^\bot}(\xi,\eta), X).
\end{eqnarray*}
From this the claim follows. By the theory of ODEs, if $b^\bot=0$ at $t=0$ then $b^\bot=0$ for all $t\ge0$.
\end{proof}

Let a family of metrics $g_t$ ($0\le t<\eps$) on $(M,\calf)$ satisfy the equality
\begin{equation}\label{E-sgeneral}
 \dt g=s(g)\,\hat g,
\end{equation}
where $s(g)$ can be considered as a $t$-dependent function on $M$.
Notice that the volume form $\vol_t$ of $g_t$ is evolved as $({d}/{dt})\vol_t=(n/2)\,s(g_t)\vol_t$,
see~\cite{rw-m}.

The proof of the next lemma is based on (\ref{eq2}) with $S(g)=s(g)\,\hat g$.

\begin{lemma}[see \cite{rw-m} and \cite{rovwol}]\label{L-btAt2}
For a totally geodesic foliation with (\ref{E-sgeneral}) and $N\in D_\calf$, we have
\begin{eqnarray}\label{Es-S-b}
 \dt b\eq s\,b -\frac12\,\hat g\,\nabla^\calf s, \qquad \dt T= 0, \\
\label{Es-S-A}
 \dt A_N\eq -\frac12\,N(s)\,\widehat\id,\qquad
 \dt T^\sharp_N = -s\,T^\sharp_N,\qquad
 \dt C_N = -\frac12\,N(s)\,\widehat\id -s\,T^\sharp_N,\\
\label{E-nablaNNt2s}
 \dt(\|T\|^2) \eq -2\,s\,\|T\|^2,\qquad
 \dt H = -\frac n2\,\nabla^\calf s,\qquad
\dt (\Div^\calf H) = -\frac{n}2\,\Delta_\calf s.
\end{eqnarray}
\end{lemma}

\begin{remark}\rm
For any function $f\in C^1$ and $N\in D_\calf$, using $(\dt g)(\cdot, N)=0$, we find
\[
 g(\nabla^\calf(\dt f), N) = N(\dt f)=\dt N(f)=\dt g(\nabla^\calf f, N) = g(\dt(\nabla^\calf f), N).
\]
Lemma~\ref{L-btAt2} and $\nabla^\calf (\dt\log\|T\|)=-\nabla^\calf s=\frac 2n\,\dt H$
yield the following conservation law for the evolution (\ref{E-sgeneral}) on the domain $U_T$:
 $\dt\big(2\,H -n\nabla^\calf\log\|T\|\big)=0$.
\end{remark}

\begin{proposition}[\textit{Conservation of ``non-umbilicity"}]\label{R-consumb}
Let $g_t\ (t\ge0)$ be a $D$-conformal family of
metrics on a foliated manifold $(M,\calf)$. Then the function $\beta_D$ doesn't depend on $t$.
\end{proposition}

\begin{proof}
Using Lemma~\ref{L-btAt2}, we calculate
\begin{eqnarray*}
 \dt\,\|b\,\|^2_{g} \eq \dt\sum\nolimits_\alpha\tr(A_{\eps_\alpha}^2)
 =2\sum\nolimits_\alpha\tr(A_{\eps_\alpha}\dt A_{\eps_\alpha})
 =-\sum\nolimits_\alpha \eps_\alpha(s)\tr A_{\eps_\alpha} = -g(\nabla s, H),\\
 \dt g(H,H) \eq s\,\hat g(H,H) +2g(\dt H, H) = -n\,g(\nabla s, H).
\end{eqnarray*}
Hence, $n\,\dt\beta_D=\dt\,\|b\,\|^2_{g} -\frac1n\,\dt g(H,H)=0$.
\end{proof}

If one has a solution $u_0$ to a given non-linear PDE, it is possible to linearise the equation by considering a smooth family $u=u(t)$ of solutions with a variation $v=\dt u_{|\,t=0}$.
Differentiation the PDE by $t$, yields a linear PDE in terms of $v$.
The next lemma concerns the linearisation of the differential operator
$g\to-2\,(\Sc_{\,\rm mix} -\|T\|^2-\Phi)\,\hat g$,
see~(\ref{E-GF-Ricmix-mu}).

\begin{lemma}\label{L-dtRStrunc}
For a totally geodesic foliation with (\ref{E-sgeneral}),
the mixed scalar curvature is evolved by
\begin{equation}\label{E-dtR-s}
 \dt(\Sc_{\,\rm mix} -\|T\|^2)= -\frac n2\,\Delta-\calf s +\nabla_H\,s.
\end{equation}
\end{lemma}

\begin{proof} Differentiating (\ref{E-RicNs}) by $t$, and using $\dt\beta_D=0$ of
Proposition~\ref{R-consumb} and $\hat g(H, H)=0$, we obtain
\begin{equation*}
 \dt(\Sc_{\,\rm mix}  -\|T\|^2) = \dt(\Div^\calf H) -\frac2n\,g(\dt H, H).
\end{equation*}
By the above, using Lemma~\ref{L-btAt2}, we rewrite the above equation as (\ref{E-dtR-s}).
\end{proof}

The following proposition shows that (\ref{E-sgeneral}) preserves certain geometric properties
of~$D$.

\begin{proposition}\label{P-08b}
Let $\calf$ be a totally geodesic foliation of a Riemannian manifold $(M, g_0)$,
and $g_t$ be a family of Riemannian metrics (\ref{E-sgeneral}) on $M$.
If $D$ is either totally umbilical, or harmonic, or totally geodesic with respect to $g_0$ then $D$ is the same for all $g_t$.
\end{proposition}

\begin{proof} If $D$ is $g_0$-umbilical then we have $b=H\,\hat g$ at $t=0$,
where $H$ is the mean curvature vector of $D$.
Applying to (\ref{Es-S-b})$_1$ the theorem on existence/uniqueness of a solution of ODEs,
we conclude that $b_t=\tilde H_t\,\hat g_t$ for all $t$, for some $\tilde H_t\in\Gamma(D_\calf)$.
Tracing this, we see that $\tilde H_t$ is the mean curvature vector of $D$ w.\,r.\,t. $g_t$,
hence $D$ is umbilical for any $g_t$.
The proof of other cases is similar.
\end{proof}

Assume that $\int_0^\infty u_0(t)\,dt<\infty$, where $u_0(t)=\sup_M |s(g_t)|_{\,g(t)}$.
Then the metrics (\ref{E-sgeneral}) are uniformly equivalent, i.e., there exists a constant $c>0$ such that
$c^{-1}\|X\|^2_{g_0} \le \|X\|^2_{g_t}\le c\,\|X\|^2_{g_0}$
for all points $(x,t)\in M\times[0,\infty)$ and all vectors $X\in T_x M$.

We will use the following condition for proving convergence of evolving metrics, see \cite{ck1}.

\begin{proposition}\label{P-converge}
$\pi:M\to B$ be a fiber bundle with compact totally geodesic fibers of a closed Riemannian manifold $(M,g_0)$.
Suppose that $g_t\ (t\ge0)$ is the solution of (\ref{E-sgeneral}).
 Define functions $u_j(t)=\sup_M |(\nabla^{t,\calf})^j s(g_t)|_{g(t)}$ and assume that $\int_0^\infty u_j(t)\,dt<\infty$ for all $j\ge0$.
 Then, as $t\to\infty$, the metrics $g_t$ converge in $C^\infty$-topology to a smooth Riemannian metric.
\end{proposition}

\subsection{Evolving of geometric quantities}
\label{subsec:evolv}

By Proposition~\ref{R-consumb}, the measure of non-umbilicity of $D$ (see Introduction),
is preserved by (\ref{E-GF-Ricmix-mu}).

From Lemmas~\ref{L-btAt2} and \ref{L-dtRStrunc} with $s=-2(\Sc_{\,\rm mix} -\|T\|^2-\Phi)$,
we obtain the following.

\begin{lemma}\label{L-dt_Cfin}
For a totally geodesic foliation with (\ref{E-GF-Ricmix-mu}), we have $\dt T=0$ and
\begin{eqnarray*}
 &&\dt b = -2\,(\Sc_{\,\rm mix} -\|T\|^2-\Phi)\,b +\hat g\,\nabla^\calf(\Sc_{\,\rm mix} -\|T\|^2),\quad
 \dt A_N = N(\Sc_{\,\rm mix} -\|T\|^2)\,\hat\id,\\
 &&\dt(\Sc_{\,\rm mix} -\|T\|^2) = n\,\Delta_\calf(\Sc_{\,\rm mix} -\|T\|^2)
 -2\,\nabla_H(\Sc_{\,\rm mix} -\|T\|^2),\\
 &&\dt(\|T\|^2) = 4\,(\Sc_{\,\rm mix} -\|T\|^2-\Phi)\,\|T\|^2.
\end{eqnarray*}
\end{lemma}

\begin{example}[\textbf{$n=1$}]\label{Ex-dt_rN}\rm
Consider a surface $(M^2, g)$ with a geodesic unit vector field $N$.
Let $k,\,K\in C^2(M)$ be the curvature of $N^\perp$-curves and the gaussian curvature of $M^2$, respectively.
We have
\[
 C(X)=k\cdot X,\qquad
 R_N(X)=K\cdot X
 \quad{\rm for}\qquad X\perp N.
\]
Take $\Phi=0$.
The equation (\ref{E-GF-Ricmix-mu}) takes the form (\ref{E-GF-KM2}).
By Lemma~\ref{L-dt_Cfin} we obtain the PDEs
\begin{equation}\label{E-Nlambda}
 \dt K=N(N(K)) -2\,k\,N(K),\qquad
 \dt k=N(K).
\end{equation}
For $n=1$, (\ref{E-riccati}) reads as the Riccati equation
\begin{equation}\label{En1-riccati}
 N(k)=k^2+K.
\end{equation}
Substituting $K$ from (\ref{En1-riccati}) into (\ref{E-Nlambda})$_2$,
we obtain the Burgers equation
\begin{equation}\label{E-kNN}
 \dt k=N(N(k))-N(k^2),
\end{equation}
it also follows from (\ref{Ec-ANtau1}) with $\beta_D=0$.
If the solution $k_t$ of (\ref{E-kNN}) is known, then by (\ref{En1-riccati}) we find $K_t=N(k_t)-k_t^2$.
Finally, we recover the metric by
 $\hat g_t=\hat g_0\exp\,(-2\!\int_0^t K_t\,dt)$.
\end{example}

\section{Proof of main results and applications}
\label{sec:egf}

\subsection{The behavior of the mean curvature vector $H$}
\label{sec:burgers}

Let $(F,g)$ be a Riemannian manifold, e.g., a leaf of $\calf$, or a fiber of $\pi:M\to B$.

\begin{example}\rm
 Consider the Burgers equation for a potential vector field $H$ on $(F,g)$
\[
 \partial_\tau H+a\,\nabla g(H,H) =\nu\,\nabla(\Div\,H)
\]
where $a\in\RR$ and $\nu>0$.
Using the scaling of independent variables $x=z\,\frac{a}{\nu}$ and $t=\tau\,\frac{a^2}{\nu}$,
the above equation reduces~to the normalized Burgers equation
\begin{equation}\label{E-Burgers}
 \dt H +\frac12\,\nabla g(H,H)=\nabla(\Div\,H).
\end{equation}
To show this, we compute
\[
 \partial_\tau H=(a^2/\nu)\,\dt H,\quad
 \nabla_z\,H=(a/\nu)\,\nabla_x\,H,\quad
 (\nabla\Div)_{z}\,H=(a^2/\nu^2)\,(\nabla\Div)_{x}\,H.
\]
Solutions of (\ref{E-Burgers}) correspond to solutions of
the homogeneous \textit{heat equation} on $(F, g)$,
\begin{equation}\label{E-heat}
 \dt u=\Delta\,u,
\end{equation}
using the well-known Cole-Hopf transformation $H=-2\,\nabla(\log u)$.
\end{example}

The \textit{forced (or inhomogeneous) Burgers equation}, see \cite{ks2001}, \cite{s2008},
has attached some attention as an analogue of the Navier-Stokes equations.
For a potential vector field $H$ and a function $f\in C^\infty(F)$,
it can be viewed as the following equation on $F$:
\begin{equation}\label{E-Burgers-in}
 \dt H +\frac12\,\nabla g(H,H)=(\nabla\Div)\,H -2\,\nabla f.
\end{equation}
Since the function $f$ is defined modulo a constant, one may assume $f\ge0$.

By the maximum principle, see \cite{ck1}, we have the following.

\begin{lemma}\label{L-u-max}
Let $f\in C^1(F)$ be an arbitrary function and $u_0\in C^2(F)$
on a closed Riemannian manifold $(F,g)$.
Then the Cauchy's problem for the heat equation with a linear reaction term
\begin{equation}\label{E-heat-in}
 \dt u =\Delta u +f\,u,\qquad
 u(\cdot,0)=u_0,
\end{equation}
has a unique solution $u(\cdot,t)\ (t\ge0)$,
and if $u(\cdot,0)\ge c$ for some $c\in\RR$ then $u(\cdot,t)\ge c$ for all~$t$.
\end{lemma}

Let $\lambda_0$ be the smallest eigenvalue (with a positive eigenfunction $e_0$)
of the Schr\"{o}dinger operator $\mathcal{H}=-\Delta -f\id$ on~$F$.
One may find the estimate $\lambda_0\ge-\max f$.

\begin{proposition}\label{T-CH}
Let $(F, g)$ be a closed Riemannian manifold, and $f\in C^\infty(F)$.

(a) If $u(x,t)$ is any positive solution of the linear PDE (\ref{E-heat-in})$_1$ on $F$
then the vector field  $H=-2\,\nabla(\log u)$ solves (\ref{E-Burgers-in}). Every solution of (\ref{E-Burgers-in}) comes by this way.

(b) Let $u(\cdot,t)\ (t\ge0)$ be a global solution of (\ref{E-heat-in}) on $F$ with $u_0>0$ and $f\ge0$.
Then $u(\cdot,t)>0$ for all $t\ge0$, and the solution vector-field $H=-2\,\nabla(\log u)$ of (\ref{E-Burgers-in}) approaches exponentially as $t\to\infty$ to a smooth vector-field $\bar H=-2\,\nabla(\log e_0)$ on $F$ -- a unique conservative solution of the~PDE
\begin{equation}\label{E-barv}
 \Div\bar H = g(\bar H,\bar H)/2 +2\,(f +\lambda_0).
\end{equation}
\end{proposition}

\begin{proof}
(a) Notice that $\dt\circ\nabla=\nabla\circ\dt$. We rewrite (\ref{E-Burgers-in}) in a form of a conservation law,
\[
 \dt H = \nabla\big(\Div H -g(H,H)/2 -2\,f\big).
\]
This can be regarded as the compatibility condition for a function $\psi$ to exist, such that
\begin{equation}\label{E-HC2}
 \nabla\psi = H,\qquad \dt\psi = \Div H -g(H,H)/2 -2\,f.
\end{equation}
Substituting $H$ from (\ref{E-HC2})$_1$ into (\ref{E-HC2})$_2$,
and using the definition $\Delta=\Div\nabla$ for functions, we obtain the following PDE:
 $\dt\psi +g(\nabla\psi,\nabla\psi)/2=\Delta\psi -2\,f$.
Next we introduce $\psi=-2\,\log u\,$ so that
\[
 \dt\psi+\frac12\,g(\nabla\psi,\nabla\psi)-\Delta\psi +2\,f
 =-\frac{2}{u}\,\big(\dt u -\Delta u -f\,u\big).
\]

(b) Using Fourier method and Theorem~\ref{T-basis}
(Section~\ref{R-burgers-heat}), we represent a solution of (\ref{E-heat-in}) as series
\begin{equation}\label{E-series-u}
 u(\cdot,t)=\sum\nolimits_{j\,\ge\,j_0} c_j\,e^{-\lambda_j\,t} e_j,\qquad c_{j_0}\ne0
\end{equation}
 by eigenfunctions of $\mathcal{H}$.
The terms with $e^{-\lambda_{j_0} t}$ in (\ref{E-series-u}) dominate as $t\to\infty$,
and can be represented in one-term form as $\tilde c\,e^{-\lambda_{j_0} t}\,\tilde e$,
where $\tilde c\ne0$ and the eigenfunction $\tilde e$ (for $\lambda_{j_0}$) has unit $L_2$-norm.
By~the maximum principle (see Lemma~\ref{L-u-max}) we conclude that $u>0$ for all $t\ge0$.
Hence $\tilde e>0$.

The eigenspace, corresponding to $\lambda_{j_0}$, is one-dimensional, see \cite[Theorem~4.8]{lp94}.
Moreover, from \cite[Chapt. 2, Theorem 2.13]{k64}
we conclude that $j_0=0$,
hence $\lambda_{j_0}=\lambda_0$ and $\tilde e=e_0$.
Since the series (\ref{E-series-u}) converges absolutely and uniformly for all $t$, there exists
the limit vector field $\lim\limits_{t\to\infty} H(\cdot,t) =\bar H$,
\[
 \bar H(x) =-2\lim\limits_{t\to\infty} \frac{\nabla u(x,t)}{u(x,t)}
 =-2\sum\nolimits_{j\ge 0} c_j\,e^{-\lambda_j\,t}\nabla e_j(x)\Big/\sum\nolimits_{j\ge 0} c_j\,e^{-\lambda_j\,t} e_j(x)
 =-2\,\frac{\nabla e_0(x)}{e_0(x)},
 \]
see (\ref{diff}) in Section~\ref{R-burgers-heat}, and convergence to $\bar H(x)$ is exponential.
By the above we find
\[
 \Div\bar H = -2\,\Delta(\log e_0)
 =-2\,(\Delta e_0)/e_0 +2\,g(\nabla e_0, \nabla e_0)/g(e_0, e_0).
\]
Hence $\Div\bar H -g(\bar H,\bar H)/2=-2\,(\Delta e_0)/e_0= 2\,(f +\lambda_0)$,
that proves (\ref{E-barv}).

To prove uniqueness, assume the contrary, that (\ref{E-barv}) has another conservative solution
$\tilde H=\nabla(-2\log\tilde e_0)$, where $\tilde e_0$ is a positive function of unit $L_2$-norm.
Next, we calculate (\ref{E-barv}):
\begin{equation*}
 \Div\tilde H -\frac12\,g(\tilde H,\tilde H) -2\,(f(x) +\lambda_0) =-\frac2{\tilde e_0}\big[\Delta\tilde e_0 +(f+\lambda_0)\,\tilde e_0\big]
\end{equation*}
and get $\Delta\tilde e_0 +(f+\lambda_0)\,\tilde e_0=0$.
Since $\lambda_0$ is a simple eigenvalue of $\mathcal{H}$,
we have $\tilde e_0=e_0$ and $\tilde H=\bar H$.
\end{proof}

\subsection{Proofs}
\label{subsec:proofs}

\begin{proof} (of \textbf{Proposition~\ref{T-main-loc}})
 Let $g=g_0+h$, where $h=s\,\hat g_0$ and $s\in C^1(M)$.
 Notice that $\nabla^\calf\,\hat g_0=0$.
 By Lemma~\ref{L-dtRStrunc}, the linearization of (\ref{E-GF-Ricmix-mu}) at $g_0$
is the linear PDE on the leaves:
\begin{equation*}
 \dt s =n\,\Delta_\calf\,s-2\,g(\nabla s,\, H(g_0)) -2\,(\Sc_{\,\rm mix} -\|T\|^2-\Phi)_{g_0} s,
 -\|T\|^2-n\,\lambda_0)_{g_0}\,h.
\end{equation*}
and $s_{|\,t=0}$ is bounded. The result follows from the theory of linear parabolic PDEs.
\end{proof}

\begin{proof} (of \textbf{Proposition~\ref{T-mainA}})
By Proposition~\ref{T-main-loc}, there exists a local solution $g_t\ (0\le t<\eps)$ to (\ref{E-GF-Ricmix-mu}).

(i) By (\ref{E-nablaNNt2s})$_2$ of Lemma~\ref{L-btAt2}
with $s=-2(\Sc_{\,\rm mix} -\|T\|^2-\Phi)$,
and using (\ref{E-RicNs}), we obtain (\ref{Ec-ANtau1}) for $H$.

(ii) As in the proof of Proposition~\ref{T-CH} for $H=\nabla^\calf\psi$,
 we reduce (\ref{Ec-ANtau1}) to
\[
 \dt\psi+g(\nabla^\calf\psi, \nabla^\calf\psi)-n\,\Delta_\calf\psi =-n^2\,\beta_D.
\]
Then applying $\psi=-n\log u$ for a positive function $u$, we calculate
\[
 \dt\psi=-\frac{n}{u}\,\dt u,\quad
 \nabla^\calf\psi=-\frac{n}{u}\,\nabla^\calf u,\quad
 \Delta_\calf\psi=-\frac{n}{u}\,\Delta_\calf u +\frac{n}{u^2}\,g(\nabla^\calf u, \nabla^\calf u),
\]
and obtain
\[
 \dt\psi+g(\nabla^\calf\psi, \nabla^\calf\psi) -n\,\Delta_\calf\psi +n^2\,\beta_D
  =-\frac{n}{u}\,\big(\dt u -n\,\Delta_\calf u -n\,\beta_D\,u\big).
\]
Thus, the function $u>0$, introduced by $H=-n\,\nabla^\calf(\log u)$, obeys the Schr\"{o}dinger equation
\begin{equation}\label{Ec-dtrho1}
 \dt u = -n\,\mathcal{H}(u).
\end{equation}
Again, by the proof of Proposition~\ref{T-CH}(b), the global solution of (\ref{Ec-ANtau1}), $H=-n\,\nabla^\calf(\log u)$, approaches exponentially as $t\to\infty$ to
$\bar H=-n\,\nabla^\calf\log e_0$, where $e_0>0$ is
the (unique) eigenfunction corresponding to the smallest eigenvalue, $\lambda_0$,
of the Schr\"{o}dinger operator $\mathcal{H}$ on compact leaves.
\end{proof}

\begin{proof} (of \textbf{Theorem~\ref{T-main0}})
(i) By Proposition~\ref{T-mainA}(i), the Cauchy's problem for (\ref{Ec-dtrho1})
with $u(\cdot,0)=u_0$ admits a unique solution $u(x,t)\ (t\ge0)$ on any fiber.
By the maximum principle (see Lemma~\ref{L-u-max}) we conclude that $u>0$ for all $t\ge0$.
Hence there exists a unique global solution $g_t$ to~(\ref{E-GF-Ricmix-mu}).

(ii) Let  $\Phi=n\,\lambda_0$.
By Proposition~\ref{T-mainA}(ii), the limit vector field $\bar H$ depends only on $[g_0]$, and is the unique stationary solution of (\ref{Ec-ANtau1}):
\begin{equation}\label{E-barH}
 \Div^\calf\bar H=g(\bar H,\bar H)/n +n\,(\beta_D+\lambda_0).
\end{equation}
We have $\lambda_0\ge-\max\limits_{M}\beta_D$ and $\nabla^\calf\lambda_0=0$.
By Lemma~\ref{L-CC-riccati} and (\ref{E-barH}), we obtain
\[
 \lim\limits_{t\to\infty}(\Sc_{\,\rm mix}  -\|T\|^2) =\Div^\calf\bar H -g(\bar H, \bar H)/n -n\,\beta_D =n\,\lambda_0.
\]
Since the convergence $\Sc_{\,\rm mix}-\|T\|^2\to n\,\lambda_0$ as $t\to\infty$ has the exponential velocity $e^{(\lambda_0-\lambda_1)\,t}$,
by Proposition~\ref{P-converge}, a unique global solution $g_t$ of (\ref{E-GF-Ricmix-mu})
converge in $C^\infty$-topology as $t\to\infty$ to a  smooth Riemannian metric~$\bar g$.
By the above, $\|T\|_{t}\ge{c}^{-2}\|T\|_{0}$ for some $c>0$.
Thus, if $\tilde C={n^2}\,c^4$ then
\[
 \overline{\rm Sc}_{\,\rm mix}=\|T\|^2_{\bar g}+n\,\lambda_0
 \ge c^{-4}\|T\|_{0}^2 -n^2\max\limits_{F(x)}\beta_D>0.
\]
By $\dt T=0$, see (\ref{Es-S-b})$_2$,
we have $\bar T(q)\ne0$ for $q\in U_T$ (remark that $U_T$ is $t$-independent).
\end{proof}

\begin{proof} (of \textbf{Corollary~\ref{C-01}})
(a)~By the above, if $\nabla^\calf\beta_D=0$
then $e_0=const$ on the fibers and $\lambda_0=-\beta_D$, hence $\bar H=0$.
 (b)~In particular, $\lambda_0=0$ when $\beta_D=0$, hence $\bar b=0$.
 By (\ref{E-nablaT-1b}) and the conditions, we have $\nabla^\calf\|T\|^2_{\bar g}=0$.
 Since $\overline{\Sc}_{\,\rm mix}=\|T\|^2_{\bar g}$, the mixed scalar curvature of $\bar g$ is constant on the fibers.
 For $T=0$, $M$ splits along $D$ and $D_\calf$ (de-Rham decomposition).
\end{proof}

\begin{proof} (of \textbf{Theorem~\ref{T-twisted1}})
By Proposition~\ref{P-08b}, the flow (\ref{E-GF-Ricmix-mu}) preserves the twisted product structure.

(i) $\Rightarrow$ (ii), (iii):
(other implications can be shown similarly).
By~Proposition~\ref{T-mainA} with $\beta_D=0$, the mean curvature $H$
of the fibers $M_1\times\{y\}$ satisfies (\ref{E-multi-B}).
By~(\ref{Ec-dtrho1}) with $\beta_D=0$ and $H=-n\,\nabla^\calf(\log f)$, we obtain
 $\dt f = n\,\Delta_\calf f$, that is the heat equation for the function $f>0$ on the fibers.
\end{proof}

\begin{proof} (of \textbf{Corollary~\ref{T-main2}})
We apply Theorem~\ref{T-main2} for the fiber bundle $\pi: M_1\times M_2\to M_1$
with totally geodesic fibers $F_x=\{x\}\times M_2$
and the potential function $\psi_0=-\log f$ (at $t=0$).
As in the proof of Theorem~\ref{T-twisted1}
(see also Section~\ref{sec:burgers}), we reduce (\ref{E-multi-B}) for $H$ to the heat equation
for $f$ along the fibers, and conclude that $\bar H=0$ for the limit metric $\bar g$.
Since the canonical foliation $M_1\times\{y\}$ is $\bar g$-umbilical, by the above we have $\bar b=0$ (i.e., $M_1\times\{y\}$ is $\bar g$-totally geodesic).
Thus, $M_1\times M_2$ is the metric product with respect to $\bar g=\bar f^2 g_1\times g_2$ (de-Rham decomposition).
\end{proof}

\subsection{Applications to surfaces}
\label{sec:surfrev}

The results of the section can be easily generalized for hypersurfaces of revolution.

\begin{example}\rm
Let $\pi: M^2\to B$ be a fiber bundle of a two-dimensional torus $(M^2, g_0)$
with Gaussian curvature $K$, and the fibers are closed geodesics.
Let the geodesic curvature $k$ of orthogonal
(to fibers) curves obeys $k=N(\psi_0)$ for a smooth function $\psi_0$ on $M^2$.
By Theorem~\ref{T-main0}, the equation (\ref{E-GF-KM2}) admits a unique solution $g_t\ (t\ge0)$ converging
as $t\to\infty$ to a flat metric, and the fibers of $\pi$ compose a rational linear foliation.
\end{example}

Metric on a surface of revolution is a special class of warped products
(see Definition~\ref{D-twisted}).
The~equation (\ref{E-GF-KM2}) on a surface of revolution
provides fruitful geometrical interpretation
of the classical relation between Burgers and heat equations.
We are looking for a one-parameter family of surfaces of revolution,
which are foliated by profile curves, and the induced metric $g_t$ obeys~(\ref{E-GF-KM2}).
The profile of $M^2_0$ parameterized as in Example~\ref{Ex-surfrev}(b)
is $XZ$-plane curve $\gamma_0=[\rho(\cdot,0), 0, h(\cdot,0)]$ (the fiber),
and $\theta$-curves are circles in $\RR^3$.
Let $x$ be the natural parameter of $\gamma_t=r(\cdot,t)$, i.e.,
\begin{equation}\label{E-natural}
 (\rho_{,x})^2+(h_{,x})^2=1.
\end{equation}
Thus $N=r_{,x}$ is the unit normal to $\theta$-curves on $M_t^2$.
 The geodesic curvature, $k$, of $\theta$-curves obeys Burgers equation, while
 the  radius $\rho$ of $\theta$-curves (as Euclidean circles) satisfies the heat equation,
 see Example~\ref{Ex-surfrev}(b); both functions are related by the Cole-Hopf transformation
 $k = -(\log\rho)_{,x}$.

\noindent
 It is known that the gaussian curvature is $K=-\rho_{,xx}/\rho$, and one may assume $\rho>0$ for $t=0$.
Notice that (\ref{En1-riccati}), $k_{,x}=k^2+K$ for $t$, is satisfied.
  The induced metric on $M_t$ has the rotational symmetric form
 $g_t =dx^2+\rho^2\,d\theta^2$.
The equation PDE for metrics reads as
 $\dt g =-2\,K\,\hat g\ \Longrightarrow\ \dt \rho = -K\,\rho$.
Thus, (\ref{E-GF-KM2}) yields the Burgers equation (\ref{E-kNN}) for $k$
and the heat equation for $\rho$,
\begin{equation}\label{E-lambdaphi}
 k_{,t} = k_{,xx}-(k^2)_{,x},\qquad
 \rho_{,t} = \rho_{,xx}.
\end{equation}
Differentiating (\ref{E-lambdaphi})$_2$ by $x$, we find $(\rho_{,x})_{,t} = (\rho_{,x})_{,xx}$.
Since $|\rho_{,x}|\le1$ for $t=0$, by the maximum principle, $|\rho_{,x}|\le1$ holds for $t\ge0$.
When such a solution $\rho\ (t\ge0)$ is known, we find $h$ from (\ref{E-natural}) as
 $h = \int \sqrt{1-(\rho_{,x})^2}\,dx$.
For example, suppose that the boundary conditions are
 $\rho(0,t)=\rho_0,\
 \rho(l,t)=\rho_1,\
 h(0,t)=z_0\ (t\ge0)$,
where $\rho_1>\rho_0>0$. By the theory of heat equation,
the solution $\rho$ approaches as $t\to\infty$ to a linear function
$\bar\rho=x\rho_0+(l-x)\rho_1>0$.
Also, $h$ approaches as $t\to\infty$ to a linear function $\bar h=x z_0+(l-x)z_1$,
where $z_1$ may be calculated from the equality $(\rho_1-\rho_0)^2+(z_1-z_0)^2=l^2$.
The curves $\gamma$ are isometric each to other for all $t$ (with the same arc-length parameter $x$).
The limit curve $\lim\limits_{t\to\infty}\gamma_t=\bar\gamma=[\bar\rho, \bar h]$
is a line segment of length $l$. Thus, $M_t$ approach as $t\to\infty$ to the
flat surface of revolution $\bar M$ -- the patch of a cone generated by $\bar\gamma$.

\section{Appendix: Parabolic PDEs on closed Riemannian manifolds}
\label{R-burgers-heat}

Let $(F^p, g)$ be a $C^\infty$-smooth closed (i.e., compact without a boundary) Riemannian manifold.
If $\mathcal{H}$ is a bounded linear operator acting from a Banach space $E_1$ to a Banach space $E_2$,
we shall write $\mathcal{H}:\;E_1\rightarrow E_2$.
The \textit{resolvent set} of $\mathcal{H}:\;E\rightarrow E$, is defined by
 $\rho(\mathcal{H})=\{\lambda\in\CC:\ \mathcal{H}-\lambda\,\id\mbox{ is invertible and } (\mathcal{H}-\lambda\id)^{-1}\ \mbox{is bounded}\}$.
The \textit{resolvent} of $\mathcal{H}$ is the operator
$R_\lambda(\mathcal{H})=(\mathcal{H}-\lambda\id)^{-1}$ for $\lambda\in\rho(\mathcal{H})$,
and the \textit{spectrum} of $\mathcal{H}$ is the set $\sigma(\mathcal{H}):=\CC\setminus\rho(\mathcal{H})$,
see \cite[Chapt. VII, Sect. 9]{ds}.
Let $H^l(F)$ be the Hilbert space of differentiable by Sobolev real functions on a manifold $F$,
of order $l$; with the inner product $(\cdot,\cdot)_{l}$ and the norm $\|\cdot\|_l$.
In particular, $H^0(F)=L_2(F)$ with the inner product $(\cdot,\cdot)_{0}$ and the norm $\|\cdot\|_0$.
We shall denote $\|\cdot\|_{c^k}$ the norm in $C^k(F)$ ($\|\cdot\|_c$ when $k=0$).
Consider the following operator acting in the Hilbert space $L_2(F)$:
\begin{equation}\label{dfH}
 \mathcal{H}(u)=-\Delta u -f(x)\,u,
\end{equation}
defined on the domain $\mathcal{D}=H^2(F)$.
The operator $\mathcal{H}$ is self-adjoint, bounded from below
(but it is unbounded). Its resolvent is compact,
i.e., for some $\lambda\in\rho(\mathcal{H})$ the operator $R_\lambda(\mathcal{H})$
maps any bounded in $L_2(F)$ set onto a set, whose closure is compact in $L_2(F)$.

\begin{proposition}[\textbf{Elliptic regularity}, see \cite{bes}]\label{prellreg}
If the operator $\mathcal{H}$ is defined by (\ref{dfH}) and
$\gamma\notin\sigma(\mathcal{H})$, then for any nonnegative integer $k$ we have
 $(\mathcal{H}-\gamma\id)^{-1}:\; H^k(F)\rightarrow H^{k+2}(F)$.
\end{proposition}

\begin{proposition}[\textbf{Sobolev embedding Theorem}, see \cite{bes}]\label{prembed}
If a nonnegative $k\in\ZZ$ and $l\in\NN$ are such that $2\,l>p+2\,k$,
then $H^l(F)$ is continuously embedded into $C^k(F)$.
\end{proposition}

\begin{proposition}\label{L-spectrL}
The spectrum $\sigma(\mathcal{H})$ consists of an infinite sequence of isolated real eigenvalues
$\lambda_0\le\lambda_1\le\lambda_2\le\dots\lambda_n\le\dots$
(counting their multiplicities), $\lambda_n\rightarrow\infty$ as
$n\rightarrow\infty$. If we fix the orthonormal basis $\{e_n\}$ in $L_2(F)$
of the corresponding eigenfunctions
(i.e., $\mathcal{H}e_n=\lambda_ne_n$, $\|e_n\|_{0}=1$),
then any function $u\in L_2(F)$ is expanded into the series
(converging to $u$ in the $L_2(F)$-norm)
\begin{equation}\label{expan}
 u({x})=\sum\nolimits_{n=0}^\infty c_n\,e_n({x}),
 \qquad
 c_n = (u, e_n)_{0} =\int_{F} u(x)\,e_n(x)\,dx.
\end{equation}
\end{proposition}

The claim of Proposition~\ref{L-spectrL} follows from the following facts. Since by Proposition~\ref{prellreg},
we have $(\mathcal{H}-\gamma\id)^{-1}:\; L_2(F)\rightarrow H^2(F)$ for $\gamma\notin\sigma(\mathcal{H})$,
and the embedding of $H^2(F)$ into $L_2(F)$ is continuous and compact, see \cite{bes},
then the operator $(\mathcal{H}-\gamma\id)^{-1}:\,L_2(F)\rightarrow L_2(F)$ is compact.
This means that the spectrum $\sigma(\mathcal{H})$ of the operator $\mathcal{H}$ is discrete, hence by the spectral expansion theorem for self-adjoint operators, the functions $\{e_n\}_{n\ge0}$ form an orthonormal basis in $L_2(F)$, see \cite[Part I, Chapt VII, Sect. 4, and Part II, Chapt. XII, Sect.~2]{ds}.

\begin{example}\label{Ex-S1}\rm
The~Cauchy's problem (\ref{E-heat}) with $u(0,\cdot)=u_0\in H^2(F)$
has a unique solution in the class of functions  $C([0,\infty),\,H^2(F))\cap C^1((0,\infty],\,L^2(F))$.
The solution has the property $u(\cdot,t)\in C^\infty(F)$ for $t>0$.
Moreover,
 $\lim\limits_{t\to\infty}u(\cdot,t)=\bar u_0=\frac1{(2\pi)^p}\int_{F} u_0(x)\,dx$
 and
 $\|u_t-\bar u_0\|\le e^{-t}\|u_0-\bar u_0\|$ for $t>0$.
The \textit{eigenvalue problem} $-\Delta u=\lambda\,u$ on $(F,g)$ has solution with a sequence of eigenvalues
with repetition (each one as many times as the dimension of its finite dimensional eigenspace) $0=\lambda_0<\lambda_1\le\lambda_2\le\cdots\uparrow\infty$.
Let $\phi_j$ be an eigenfunction with eigenvalue $\lambda_j$, satisfying $\int_{F}\phi_i(x)\phi_j(x)\,d x_g=\delta_{ij}$.
For $\lambda_0$, the eigenfunction is the constant $\phi_0=\vol(F,g)^{-1/2}$.
\end{example}

Our goal is to formulate conditions under which this series converges uniformly
to $u$ and it is possible to differentiate it. For this we need
estimates for the eigenvalues and the eigenfunctions of~$\mathcal{H}$.
Denote the distribution function of eigenvalues of $\mathcal{H}$ by
$\mathcal{N}(\lambda)=\#\{\lambda_n: \lambda_n\le\lambda\}$.

H\"{o}rmander \cite{hor} obtained an asymptotic formula for the kernel $e(x,y,\lambda)$
of the spectral projection $E(\lambda)$ (see \cite[Part II, Chapt. XII]{ds}),
which for compact $F$ has the form
 $e(x,y,\lambda)=\sum\nolimits_{\lambda_n\le\lambda}e_n(x)e_n(y)$.
In our case, this formula is represented by
$e(x,x,\lambda)=\alpha(x)\lambda^{\frac{p}{2}}(1+o(1))$ for $\lambda\rightarrow\infty$ uniformly w.\,r.\,t. $x\in F$,
where the function $\alpha(x)$ belongs to $C^\infty(F)$ and depends only on $(F,g)$.
Integrating the formula for $e(x,x,\lambda)$ over $F$, we obtain the formula of Weyl asymptotics
\begin{equation}\label{weyl}
 \mathcal{N}(\lambda)=\theta\lambda^{\frac{p}{2}}(1+o(1))\quad \rm{as}\quad
\lambda\rightarrow\infty,
\end{equation}
where the constant $\theta>0$ depends only on $(F,g)$.

\begin{lemma}\label{lesteigf}
There exists $\delta>0$ and $\gamma_0\in\RR$
such that for any $n\in\NN\cup\{0\}$ we have $e_n\in C(F)$ and
\begin{equation}\label{esteigf}
\|e_n\|_c\le \delta(\lambda_n+\gamma_0)^{[p/4]+1}.
\end{equation}
\end{lemma}

\begin{proof}
If we take $\gamma>-\lambda_0$, then the operator $\mathcal{H}+\gamma\id$ is
invertible in $L_2(F)$ and its inverse $(\mathcal{H}+\gamma\id)^{-1}$ is bounded in $L_2(F)$.
By Proposition~\ref{prellreg},
 $(\mathcal{H}+\gamma\id)^{-1}:\;H^k(F)\rightarrow H^{k+2}(F)$ holds for $k=0,1,2,\dots$
 Then for any $l\in\NN$ we have
\begin{equation}\label{powact}
(\mathcal{H}+\gamma\id)^{-l}:\;L_2(F)\rightarrow H^{2l}(F).
\end{equation}
As is easy to check, for any nonnegative integer $n$ we have
$e_n=(\lambda_n+\gamma)^l(\mathcal{H}+\gamma\id)^{-l}e_n$. In view of (\ref{powact}),
 $e_n\in H^{2l}(F)$ holds, and we have
 $\|e_n\|_{2l}\le\tilde\delta(\lambda_n+\gamma)^l$ for some $\tilde\delta>0$ and $n=0,1,2,\dots$

 On the other hand, by Proposition~\ref{prembed} with $k=0$, for $4\,l>p$ the space
$H^{2l}(F)$
is continuously embedded into $C(F)$. Hence $e_n\in C(F)$, and we have
 $\|e_n\|_c\le\bar\delta\,\|e_n\|_{2l}$ for some $\bar\delta>0$ and $n=0,1,2,\dots$
 The above estimates
imply the desired inequality (\ref{esteigf}) with $\delta=\tilde\delta\,\bar\delta$.
\end{proof}

\begin{theorem}\label{T-basis}
Let $(F, g)$ be a closed Riemannian manifold, and $f\in C^\infty(F)$. Then
for the operator $\mathcal{H}=-\Delta u -f(x)\,\id$, see (\ref{dfH}),
any eigenfunction $e_n$ belongs to class $C^\infty(F)$, and

(i) the expansion (\ref{expan}) converges to $u$ absolutely and uniformly on $F$;

(ii) for any multi-index $\alpha$ with $|\alpha|\ge 1$ we have
\begin{equation}\label{diff}
 D^\alpha u=\sum\nolimits_{n=0}^\infty(u, e_n)_{0}D^\alpha e_n,
\end{equation}

and this series converges to $D^\alpha u$ absolutely and uniformly on $F$.
\end{theorem}

\begin{proof}
(i) Since $u\in C^\infty(F)$, for any
$m\in\NN$ and $\gamma\in\RR$ the function $h=(\mathcal{H}+\gamma\id)^m u$
is continuous on $F$, hence $h\in L_2(F)$. For $\gamma>-\lambda_0$,
the operator $\mathcal{H}+\gamma\id$ is invertible and the operator
$(\mathcal{H}+\gamma\id)^{-1}$ is defined on the whole $L_2(F)$, hence $u=(\mathcal{H}+\gamma\id)^{-m}h$. Therefore,
\begin{equation*}
 (u,\,e_n)_{0}=((\mathcal{H}+\gamma\id)^{-m}h,\,e_n)_{0}=(h,\,(\mathcal{H}+\gamma\id)^{-m}e_n)_{0}
 =(\lambda_n+\gamma)^{-m}(h,\,e_n)_{0}.
\end{equation*}
Hence in view of Lemma~\ref{lesteigf}, we get for
$l>\frac{p}{4}$ the following estimate for the terms of the series (\ref{expan}):
\begin{equation*}
\|(u, e_n)_{0}\,e_n\|_c\le \delta(\lambda_n+\gamma)^{-m+l}\|h\|_{0}.
\end{equation*}
In view of (\ref{weyl}), there exists $\delta_1>0$ such that
the counting function is estimated as $\mathcal{N}(\lambda)\le\delta_1\lambda^{\frac{p}{2}}$ for any $\lambda\in\RR$.
If we take $m>\frac{p}{2}+l$, then we get, using integration by parts
in the Stilties integral:
\begin{eqnarray*}
 \sum\nolimits_{n=0}^\infty(\lambda_n+\gamma)^{-s}
 \eq\int_{-\infty}^\infty\frac{d\,\mathcal{N}(\lambda)}{(\lambda+\gamma)^s}
 =\frac{\mathcal{N}(\lambda)}{(\lambda+\gamma)^s}\big|_{\lambda_0-1}^\infty
  +s\int_{\lambda_0-1}^\infty\frac{\mathcal{N}(\lambda)\,d\lambda}{(\lambda+\gamma)^{s+1}}\\
 \eq s\int_{\lambda_0-1}^\infty\frac{\mathcal{N}(\lambda)\,d\lambda}{(\lambda+\gamma)^{s+1}}\le
  s\,\delta_1\theta\int_{\lambda_0-1}^\infty\frac{d\lambda}{(\lambda+\gamma)^{s+1-{p}/{2}}},
\end{eqnarray*}
where $s=m-l$. The last integral converges, hence the series
(\ref{expan}) converges absolutely and uniformly on $F$. Since this
series converges to $u$ in $L_2(F)$, then it converges uniformly to $u$.

(ii) Let  $k\in\NN$ and $4\,l>p+2\,k$.
By Proposition~\ref{prembed}, the space $H^{2l}(F)$ is continuously embedded into $C^k(F)$.
As in the proof of Lemma~\ref{lesteigf}, we obtain that
there exists $\delta_k>0$ such that for any integer $n\ge0$ we have $e_n\in C^k(F)$ and
\begin{equation}\label{esteigf1}
\|e_n\|_{c^k}\le\delta_k(\lambda_n+\gamma)^l.
\end{equation}
Since $k$ is arbitrary, we conclude that any eigenfunction $e_n$ of the operator $\mathcal{H}$
belongs to class $C^\infty(F)$. Similarly as in the proof of claim (i),
for $4\,l>p+2\,k$ and $m\in\NN$, using (\ref{esteigf1}), we obtain
 $\|(u, e_n)_{0}e_n\|_{c^k}\le \delta_k(\lambda_n+\gamma)^{-m+l}\|h\|_{0}$,
 where $h=(\mathcal{H}+\gamma\id)^m u$. Hence, for $|\alpha|\le k$, we obtain
\[
  \|(u, e_n)_{0}D^\alpha e_n\|_c\le\delta_k(\lambda_n+\gamma)^{-m+l}\|h\|_{0}.
\]
 Then, as in the proof of claim (i), we obtain that if $m>\frac{p}{2}+l$,
 then the series in (\ref{diff}) converges
absolutely and uniformly. Since by claim (i), the series (\ref{expan}) converges uniformly to $u$,
then by the standard argument of Analysis, the series (\ref{diff}) converges uniformly
to the derivative $D^\alpha u$.
\end{proof}

\section*{Acknowledgment}
The authors would like to thank
Prof. I.\,Vaisman for his comments in general, and
Prof. Y.\,Pincho\-ver for sending paper \cite{lp94} to authors and helpful discussion  concerning Section~\ref{R-burgers-heat}.

\end{document}